\documentclass[11pt]{amsart}
\usepackage{appendix}
\usepackage{bbold}
\usepackage{graphicx,enumitem,dsfont}
\usepackage{verbatim}
\usepackage[usenames]{color}
\usepackage[all]{xy}
\usepackage{amsmath, mathrsfs, amssymb, amsthm}
\usepackage[colorlinks]{hyperref}
\usepackage{subfigure}
\newtheorem{theorem}{Theorem}[section]
\newtheorem{proposition}[theorem]{Proposition}
\newtheorem{lemma}[theorem]{Lemma}
\newtheorem{corollary}[theorem]{Corollary}
\newtheorem{remark}[theorem]{Remark}

\newtheorem{definition}[theorem]{Definition}
\theoremstyle{definition}

\allowdisplaybreaks
\numberwithin{equation}{section}

\newcommand{\eff}{\text{\it eff}}
\newcommand{\et}{\text{\it \'et}}
\newcommand{\Aff}{\mathbb{A}}

\title{Motives of Deligne-Mumford Stacks}

\author{Utsav Choudhury}
\address{Institut f\"ur Mathematik\\ Universit\"at Z\"urich\\ Winterthurerstrasse 190\\ CH-8057 Z\"urich\\ Switzerland}
\email{utsav.choudhury@math.uzh.ch}
\thanks{Research supported by the Swiss National Science Fundation}

\begin{document}

\begin{abstract}
For every smooth and separated Deligne-Mumford stack $F$, we associate a motive $M(F)$ in Voevodsky's category of mixed motives with rational 
coefficients $\mathbf{DM}^{\eff}(k,\mathbb{Q})$. When $F$ is proper over a field of characteristic $0$, we compare $M(F)$ with the Chow motive associated to $F$ by Toen (\cite{t}). Without the properness condition we show that $M(F)$ is a direct summand of the motive of a smooth quasi-projective variety.   
\end{abstract}

\maketitle

\tableofcontents

\section{Introduction}

The study of Deligne-Mumford stacks from a motivic perspective 
began in \cite{bm} where the notion 
of the $DMC$-motive associated to 
a proper and smooth Deligne-Mumford stack
was introduced as a tool for 
defining Gromov-Witten invariants. The construction of the category $\mathcal{M}_k^{DM}$ of $DMC$-motives uses $A^*$-Chow cohomology theories for Deligne-Mumford stacks as described in \cite{g, j1, j2, ak2, vi, eg}. These $A^*$-Chow cohomology theories coincide with rational coefficients.  
In \cite[Theorem 2.1]{t}, Toen shows that 
the canonical functor 
$\mathcal{M}_k \to
\mathcal{M}_k^{DM}$, from the category of 
usual Chow motives, is an equivalence rationally. 
In particular, to every smooth and proper Deligne-Mumford 
stack $M$, Toen associates a Chow motive $h(M)$.

In this paper we construct motives for smooth (but not necessarily proper) Deligne-Mumford stacks as objects of Voevodsky's 
triangulated category of motives
$\mathbf{DM}^{\eff}(k,\mathbb{Q})$. In the proper case, 
we compare these motives with the Chow motives we get using Toen's equivalence of categories $\mathcal{M}_k\simeq \mathcal{M}_k^{DM}$.
Without assuming properness, our construction 
of the motive of a smooth Deligne-Mumford stack $F$ seems 
to be the first one. However, 
in \cite[Thm 0.1]{gs} Gillet and Soul\'e constructed 
a motivic invariant attached to $F$, namely 
a complex of Chow motives. We hope to 
recover their invariant by applying Bondarko's weight functor  
to $M(F)$. (See \cite[Prop. 6.3.1]{bon}.)
We leave this for a future 
investigation.

The paper is organised as follows.

In section \ref{1}, we briefly review Morel-Voevodsky $\Aff^1$-homotopy category $\mathbf{H}^{\et}(k)$ and Voevodsky triangulated category of motives $\mathbf{DM}^{\eff}(k, \mathbb{Q})$. We also 
construct the functor 
$M:\mathbf{H}^{\et} (k) \to \mathbf{DM}^{\eff}(k, \mathbb{Q})$. 
Given a presheaf of small groupoids $F$, 
we associate an object $Sp(F)$ in $\mathbf{H}^{\et} (k)$.
The motive of $F$ is defined to be $M(F) := M(Sp(F))$. We then show that for a Deligne-Mumford stack $F$ and an \'etale atlas 
$u : U \to F$ we have an isomorphism $M(U_{\bullet}) \cong M(F)$ in $\mathbf{DM}^{\eff}(k, \mathbb{Q})$. Here $U_{\bullet}$ is the 
\u{C}ech hypercovering corresponding to the atlas $u: U \to F$.

In section \ref{2}, we compare the motive of a separated Deligne-Mumford stack $F$ with the motive of the coarse moduli space of $F$.
If $\pi : F \to X$ is the coarse moduli space of a separated Deligne-Mumford stack $F$, we show that the natural morphism 
$M(F) \to M(X)$ is an isomorphism in $\mathbf{DM}^{\eff}(k, \mathbb{Q})$. We then prove projective bundle formula and blow-up formula for smooth Deligne-Mumford stacks.
We also construct the Gysin triangle associated to a smooth, closed substack $Z$ of a smooth Deligne-Mumford stack $F$. 

In section \ref{3}, we show that for any smooth and separated Deligne-Mumford stack $F$ over a field of characteristic zero, 
the motive $M(F)$ is a direct factor of the motive of a smooth quasi-projective scheme. If $F$ is proper we may take this scheme to be projective.

In section \ref{4}, we show that the motivic cohomology of a Deligne-Mumford stack (see \cite[3.0.2]{j2}) is representable in $\mathbf{DM}^{\eff}(k,\mathbb{Q})$.

Finally in section \ref{5} we compare our construction with Toen's construction and prove that for any smooth and proper Deligne-Mumford stack $F$ there is a canonical isomorphism 
$\iota \circ h(F) \cong M(F)$; this is Theorem
\ref{thm:compare-toen}. 
Here $\iota : \mathcal{M}_k^{\eff} \to \mathbf{DM}^{\eff}(k,\mathbb{Q})$ is the fully faithful embedding described in \cite[Proposition 20.1]{mvw}.

The paper ends with two appendixes. 
In Appendix \ref{appendix-A}, we show that some naturally defined 
functor $\omega:\mathcal{M}_k^{\eff} \to PSh(\mathcal{V}_k)$
is fully faithful (see \ref{fully faithful}). This statement appears without proof in \cite[2.2]{ajs} and is 
mentioned in \cite[page12]{t}. It is also needed in the proof 
of \ref{thm:compare-toen}.
In Appendix \ref{appendix-B} we provide a technical result used in Appendix \ref{appendix-A}.

\medskip

\noindent
\textbf{Acknowledgements.}\,
I warmly thank D.~Rydh for his valuable suggestions. The idea of the proof of theorem \ref{quotient} was entirely communicated 
by him. I also thank A.~J.~Scholl
for discussions about \cite{ajs}
and A.~Kresch for answering questions on stacks.
This paper is part of my PhD thesis under the 
supervision of J.~Ayoub. I thank him for his 
guidance during this project. Finally, I would like to thank the referee for useful comments.

\section{The general construction}
\label{1}

In this section we describe our construction of the motive associated 
to a smooth Deligne-Mumford stack. In fact, our construction 
applies more generally to any stack but the existence of atlases 
can be used to give explicit models. We start by recalling
the motivic categories used in this paper.

\subsection{Review of motivic categories}
Let $Sm/k$ be the category of smooth separated finite type $k$-schemes and denote
$PSh(Sm/k)$ the category of presheaves of sets on $Sm/k$. Also denote $\bigtriangleup^{op} PSh(Sm/k)$ the category of spaces, i.e., presheaves of simplicial sets. As usual $\bigtriangleup$ is the category of simplices.

$\bigtriangleup^{op} PSh(Sm/k)$ has a local model structure with respect to the \'etale topology \cite[Theorem 2.4 and Corollary 2.7]{jar}. A morphism $f : \mathcal{X} \to \mathcal{Y} \in \bigtriangleup^{op} PSh(Sm/k)$ is a local weak equivalence if the induced morphisms on the stalks (for the \'etale topology) are weak equivalences of simplicial sets. Cofibrations are monomorphisms and fibrations are characterised by the right lifting property. We denote by $\mathbf{H}_s^{\et} (k)$ the homotopy category of $\bigtriangleup^{op} PSh(Sm/k)$ with respect to the \'etale local model structure, i.e., obtained by inverting formally
the local weak equivalences.

Following \cite[\S 3.2]{mv}, we consider the Bousfield localisation of the local model structure on $\bigtriangleup^{op} PSh(Sm/k)$ with respect to the class of maps $\mathcal{X} \times \mathbb{A}^1 \to \mathcal{X}$ where 
$\mathcal{X}\in \bigtriangleup^{op} PSh(Sm/k)$. The resulting model structure will be simply called the (\'etale) motivic model structure. The homotopy category with respect to this \'etale motivic model structure is denoted by $\mathbf{H}^{\et} (k)$.
(We warn the reader that in 
\cite[\S 3.2]{mv} the Nisnevich topology is used 
instead of the \'etale topology.)

\begin{remark}
Denote $(\bigtriangleup^{op} PSh(Sm/k))_{\bullet}$ the category of pointed spaces, i.e., presheaves of pointed simplicial sets on $Sm/k$. We also have the pointed versions of the local and motivic model structures where weak equivalences are detected after forgetting the pointing. The homotopy categories are denoted by $\mathbf{H}_{s, \bullet}^{\et}(k)$ and $\mathbf{H}_{\bullet}^{\et}(k)$
respectively.
\end{remark}

Now we briefly recall some facts on Voevodsky's motives. Recall that $SmCor(k)$ is the category of finite correspondences. Objects of this category are smooth $k$-schemes $X$. For $X,\,Y \in Sm/k$, 
$Hom_{SmCor(k)}(X,Y)$ is given by the group of 
finite correspondences $Cor(X,Y)$. This is the free abelian group generated by integral closed subschemes $W \subset X \times Y$ which are finite and surjective on a connected component of $X$.  Thus, if $X = \coprod_i X_i$ we have 
$Cor(X \times Y) = \bigoplus_i Cor(X_i \times Y)$.
 
A presheaf with transfers is a contravariant additive functor on $SmCor(k)$. Denote by $PST(k,\mathbb{Q})$ the category of presheaves with transfers with values in the category of $\mathbb{Q}$-vector spaces. 
A typical example is given by $\mathbb{Q}_{tr}(X)$ for $X \in Sm/k$. This presheaf 
associates to each $U \in Sm/k$ the vector space 
$Cor(U,X)\otimes \mathbb{Q}$.

Analogous to the local and motivic model category structures on the category of spaces, we have a local and a motivic model structures on the category $K(PST(k,\mathbb{Q}))$ of complexes of presheaves with transfers. A morphism $K \to L$ in $K(PST(k,\mathbb{Q}))$ is an \'etale weak equivalence if it induces quasi-isomorphisms on stalks for the \'etale topology
(or the Nisnevich topology; it doesn't matter 
in the presence of transfers). 
Cofibrations are monomorphisms and fibrations are characterized 
by the right lifting property. This gives the local model category structure on $K(PST(k,\mathbb{Q}))$ (cf. \cite[Theorem. 2.5.7]{a1}). The homotopy category is nothing but the derived category of \'etale sheaves with transfers $\mathbf{D}(Str(Sm/k,\mathbb{Q}))$. (Here $Str(Sm/k,\mathbb{Q})$ is the category of \'etale sheaves with transfers). The motivic model structure is the Bousfield localisation of the local model structure with respect to the class of morphisms $\mathbb{Q}_{tr}(\mathbb{A}^1 \times X)[n] \to \mathbb{Q}_{tr}(X)[n]$ for $X\in Sm/k$ and $n \in \mathbb{Z}$.
The resulting homotopy category with respect to the motivic model structure is denoted by $\mathbf{DM}^{\eff}(k,\mathbb{Q})$.
This is Voevodsky's triangulated category of mixed motives (with 
rational coefficients).

The functor $\mathbb{Q}_{tr}(-) : Sm/k \to PST(k,\mathbb{Q})$
 extends to a functor $\mathbb{Q}_{tr}:PSh(Sm/k) \to PST(k,\mathbb{Q})$ given by $\mathbb{Q}_{tr}(F) = Colim_{X \to F} \;\mathbb{Q}_{tr}(X)$ for any 
presheaf of sets $F$ on $Sm/k$.
In the next statement, 
$N(-)$ denotes the functor that associates the 
normalized chain complex to a simplicial object
in an additive category (cf. \cite[page 145]{gj}).

\begin{proposition}
\label{motivefunctor}
There exists a functor $M:\mathbf{H}^{\et}(k) \to \mathbf{DM}^{\eff}(k,\mathbb{Q})$. It sends a simplicial scheme $X_{\bullet}$ to
the $N\mathbb{Q}_{tr}(X_{\bullet})$.
\end{proposition}

\begin{proof}
This is well known. We give a sketch of proof here. The functor 
$\mathbb{Q}_{tr}(-):PSh(Sm/k) \to PST(k,\mathbb{Q})$ extends to a functor 
$$N\mathbb{Q}_{tr}(-):\bigtriangleup^{op} PSh(Sm/k) \to K(PST(k,\mathbb{Q})).$$ 
There is a functor $\Gamma:K(PST(k,\mathbb{Q})) \to \bigtriangleup^{op} PSh(Sm/k)$ right adjoint to $N \mathbb{Q}_{tr}$
(cf. \cite[page 149]{gj}). We will show that the pair
$(N\mathbb{Q}_{tr},\Gamma)$ is a Quillen adjunction 
for the projective motivic model structures.
(These model structures are different from the ones described above:
they have the same weak equivalences but the cofibrations
in the projective ones are defined by the left lifting 
property with respect to section-wise trivial fibrations
of presheaves of simplicial sets and surjective
morphisms of complexes of 
presheaves with transfers respectively.)

The functor $\Gamma$ takes section-wise weak equivalences in 
$K(PST(k,\mathbb{Q}))$
to section-wise weak equivalences in $\bigtriangleup^{op} PSh(Sm/k)$ and it takes surjective morphisms to section-wise fibrations in $\bigtriangleup^{op} PSh(Sm/k)$. Hence for the projective global model structures the pair $(N\mathbb{Q}_{tr}, \Gamma)$ is a Quillen adjunction. 
By \cite[Theorem 6.2]{dhi} the projective \'etale local model structure 
on $\bigtriangleup^{op} PSh(Sm/k)$ is 
the Bousfield localisation of the global projective model structure with respect to general hypercovers for the \'etale topology. Let $S$ be the class of those hypercovers. To show that 
the pair $(N\mathbb{Q}_{tr}, \Gamma)$ is a Quillen adjunction
for the \'etale local model structures, we need to show that 
the left derived functor of $N\mathbb{Q}_{tr}$ maps morphisms in $S$ to \'etale local weak equivalences in $K(PST(k,\mathbb{Q}))$. 
For this
it is enough to show that $\Gamma$ maps a local fibrant object $C_{\bullet}$ of $K(PST(k,\mathbb{Q}))$ to an $S$-local object of 
$\bigtriangleup^{op}PSh(Sm/k)$. 
Showing that 
$\Gamma(C_{\bullet})$ is $S$-local 
is equivalent to 
showing that the \'etale hypercohomology 
$\mathbb{H}^n_{\et}
(X,C_{\bullet})$ is isomorphic to
$H^n(C_{\bullet}(X))$ for any $X \in Sm/k$ and 
$n\geq 0$. Now, the hypercohomology $\mathbb{H}^n_{\et} (X, C_{\bullet})$ can be calculated using \u{C}ech hypercovers $U_{\bullet} \to X$ and moreover by \cite[Prop 6.12]{mvw} the complex $\mathbb{Q}_{tr} (U_{\bullet})$ is a resolution of the \'etale sheaf $\mathbb{Q}_{tr} (X)$. Since $C_{\bullet}$ is local fibrant we have 
$$H^n(C_{\bullet} (X)) = Hom_{Ho(K(PST(k,\mathbb{Q})))} (\mathbb{Q}_{tr} (U_{\bullet}), C_{\bullet}[n])=H^n(Tot(C_{\bullet} (U_{\bullet}))).$$ 
(Here $Ho(K(PST(k,\mathbb{Q})))$ is the homotopy category with respect to global projective model structure on $K(PST(k,\mathbb{Q}))$.) 
Now passing to the colimit over hypercovers $U_{\bullet}\to X$ we get  $\mathbb{H}^n_{\et} (X, C_{\bullet}) \cong H^n(C_{\bullet} (X))$.   

At this point we get a functor 
$\mathbf{H}^{\et}_s (k) \to \mathbf{D}(Str(Sm/k))$ 
and it remains to show that this functor takes 
the maps $\mathcal{X} \times \mathbb{A}^1 \to \mathcal{X}$
to motivic weak equivalences. This is clear by construction. 
\end{proof}

\begin{remark}
\label{remark:integral-coef-M}
There is also a functor
$M:\mathbf{H}^{\et}(k)\to \mathbf{DM}^{\eff,\,\et}(k,\mathbb{Z})$
to Voevodsky's category of \'etale motives with 
integral coefficients. It is constructed exactly as 
above. (Note that with integral coefficients, the categories 
of \'etale and Nisnevich motives are different: we denote 
them $\mathbf{DM}^{\eff,\,\et}(k,\mathbb{Z})$ and 
$\mathbf{DM}^{\eff}(k,\mathbb{Z})$ respectively;
with rational coefficient these categories 
are the same.)
\end{remark}

\subsection{The construction}
Let $Grpd$ be the category of (small) groupoids. Let $\mathcal{C}$ be any category.

Consider $2-Fun(\mathcal{C}^{op}, Grpd)$ the category of lax $2$-functors from $\mathcal{C}$ to $Grpd$. Recall that a lax $2$-functor $F$ associates to $X \in \mathcal{C}$ a groupoid $F(X)$, to $f : Y \to X$  a functor $F(f) : F(X) \to F(Y)$, and to composable morphisms $f$ and $g$ an isomorphism $F(f) \circ F(g) \cong F(g \circ f)$.  The 
$1$-morphisms between two lax $2$-functors
$\mathcal{F}$ and $\mathcal{G}$ are lax natural transformations $H$ such that for any $f: Y \to X \in \mathcal{C}$ there is a natural isomorphism between the fucntors $G(f) \circ H_X$ and $H_Y \circ F(f)$.
For any composable morphisms $f$ and $g$, we have the usual compatibility conditions.
$2$-isomorphisms between lax transformations $H$ and $H'$ are given by isomorphisms of functors $a_X : H_X \cong H'_X$ for each $X \in \mathcal{C}$, such that for
any $f: Y \to X$ we have $G(f)(a_X) = a_Y (F(f))$.

For objects $X,\, Y \in \mathcal{C}$, consider the set $Hom_{\mathcal{C}} (Y,X)$ as a discrete groupoid, i.e, all morphisms are identities. In this way, the functor
$Hom_{\mathcal{C}} (-,X):\mathcal{C} \to Grpd$ is a 
strict $2$-functor which we denote by 
$h(X)$.

\begin{lemma} 
\label{equivalence}
Let $F \in 2-Fun(\mathcal{C}^{op}, Grpd)$. There is a surjective equivalence of categories $Hom_{2-Fun(\mathcal{C}^{op}, Grpd)} (h(X), F) \to F(X)$ given by evaluating at $id_X \in h(X)(X)$.
\end{lemma}

\begin{proof}
Given any lax natural transformation $H: h(X) \to F$ we get an object $X' := H_X(id_X) \in F(X)$. Given two lax natural transformations $H,H'$ and a $2$-isomorphism $a$ between them, we get an isomorphism $a_X(id_X) : H_X(id_X) \cong H'_X(id_X)$.  Let $X' \in F(X)$. We have a  natural transformation given by $G_Y (f : Y \to X) = F(f)(X')$. Since $F$ is a lax presheaf we have $F(f \circ g)(X') \cong F(g) \circ F(f)$ for any $Z \xrightarrow{g} Y \xrightarrow{f} X$. Hence we get the required natural transformation between $F(g) \circ G_Y$ and $G_Z \circ h(X)(g)$. Moreover let $H,\,G:h(X) \to F$ such that there exists a morphism $f : H_X(id_X) \to G_X(id_X) \in F(X)$. We define a unique $2$-isomorphism $a$ between $H$ and $G$ in the following way.
For any $g \in h(X)(Y)$ we have $H_Y(g) \cong F(g)(H_X(id_X))$ given by the structure of the lax natural transformation. Similarly we get $G_Y(g) \cong F(g)(G_X(id_X))$. But then there exists $F(g) (f) :  F(g)(H_X(id_X)) \cong F(g)(G_X(id_X))$. So $a_Y(g) : H_Y(g) \cong G_Y(g)$ and $a_X (id_X) : H_X(id_X) \to G_X(id_X)$ is equal to $f$.
\end{proof}

\begin{remark} 
\label{rigid}
In general $Hom_{2-Fun(\mathcal{C}^{op}, Grpd)} (h(X), F)$ is not small unless $\mathcal{C}$ is small. Let $Sch/k$ be the category of finite type $k$-schemes.  We fix $C \subset Sch/k$ which is a full small subcategory equivalent to $Sch/k$. For any $X \in Sch/k$
and $F \in 2-Fun((Sch/k)^{op}, Grpd)$ the association $X  \mapsto Hom_{2-Fun(C^{op}, Grpd)} (h(X)|_C, F|_C)$ gives a strict presheaf of groupoids. We denote it by $h_{st} (F)$ . By \ref{equivalence} we have an
equivalence $F|_C \cong h_{st}(F)|_C$.

\end{remark}

\begin{definition}
Let  $F \in 2-Fun((Sch/k)^{op}, Grpd)$. Then the $\mathbb{A}^1$-homotopy type of $F$ is the space $Sp(F) := Ner(h_{st}(F))|_{Sm/k}$ considered as an object of $\mathbf{H}^{\et}(k)$. Here $Ner$ is the nerve functor.
\end{definition}

\begin{definition}
Let $F \in 2-Fun((Sch/k)^{op}, Grpd)$. Then the motive of $F$ is defined as $M(F):=M(Sp(F))$. This gives a functor 
$$M: 2-Fun((Sch/k)^{op}, Grpd) \to \mathbf{DM}^{\eff}(k,\mathbb{Q}).$$
Using \ref{remark:integral-coef-M}
we can also define an integral version of the motive of $F$, which 
we also denote $M(F)$ if no confusion can arise.
\end{definition}

\subsection{The case of a Deligne-Mumford stack}

Let $F:(Sch/k)^{op} \to Grpd$ be a lax $2$-functor. 

\begin{definition}
The functor $F$ is a stack in the \'etale topology if it satisfies the following axioms where
$\left\{f_i:U_i \rightarrow U\right\}_{i \in I}$ is an \'etale covering of $U\in Sch/k$ and 
$f_{ij,i}:U_i \times_U U_j \to U_i$ are the projections.

\begin{enumerate}

\item (Glueing of morphisms) If $X$ and $Y$ are two objects of $F(U)$, and $\phi_i : F(f_i)(X) \cong F(f_i) (Y)$ are isomorphisms such that $F(f_{ij,i})(\phi_i) = F(f_{ij,j}) (\phi_j)$, then there exists an isomorphism $\eta : X \cong Y$ such that $F(f_i)(\eta) = \phi_i$. 

\item (Separation of morphisms) If $X$ and $Y$ are two objects of $F(U)$, and $\phi : X \cong Y$, $\psi : X \cong Y$ are isomorphisms such that $F(f_i)(\phi) = F(f_i) (\psi)$, then $\phi = \psi$.

\item (Glueing of objects) If $X_i$ are objects of $F(U_i)$ and $\phi_{ij} : F(f_{ij,j}) (X_j) \cong F(f_{ij,i}) (X_i)$ are isomorphisms satisfying the cocycle condition 
$$(F(f_{ijk,ij})(\phi_{ij})) \circ (F(f_{ijk,jk})(\phi_{jk})) = 
F(f_{ijk,ik}) (\phi_{ik}),$$ 
then there exist an object $X$ of $F(U)$ and $\phi_i : F(f_i)(X) \cong X_i$ such that 
$\phi_{ji} \circ (F(f_{ij,i})(\phi_i)) = F(f_{ij,j})(\phi_j)$.

\end{enumerate}

\end{definition}

\begin{remark}
There is a notion of strict stacks (see \cite{sh}, \cite{j}). If $F$ is a strict presheaf of groupoids then by \cite[lemma 7, lemma 9]{j} there exists a strict stack $St(F)$ and a morphism $st : F \to St(F)$ such that $st$ is a local weak equivalence, i.e., $st$ induces equivalences of groupoids on stalks. The stack $St(F)$ is called the associated stack of $F$ and the functor is called the stackification functor.
\end{remark}

For any groupoid object $R \rightrightarrows U$ in $Sch/k$ we can associate a strict presheaf of groupoids $h(R \rightrightarrows U)$ (see the proof of lemma \ref{mine}).
\begin{definition}
A Deligne-Mumford stack $F$ is a stack on $Sch/k$ admitting a local equivalence (stalk-wise equivalence in the \'etale topology) 
$h(R \rightrightarrows U) \to F$, where  $R \rightrightarrows U$
is a groupoid object in $Sch/k$, such that both morphisms $R \to U$ are \'etale and $R \to U \times_k U$ is finite.

\end{definition}

\begin{remark}
Our definition of a Deligne-Mumford stack is equivalent to that 
of separated finite type Deligne-Mumford stack from \cite{lm}. The morphism $p : U \to F$ is representable and is called the atlas of $F$. We also have $R \cong U \times_F U$. We say that $F$ is smooth if $U$ is smooth.

\end{remark}

Given an atlas $f:U \to F$ of a smooth Deligne-Mumford stack $F$, 
we get a simplicial object $U_{\bullet}$ in $Sm/k$ by defining 
$U_i=U \times_F \cdots \times_F U$ ($i+1$ times) and the face and degeneracy maps are defined by relative diagonal and partial projections. 

\begin{lemma}
\label{mine}
For $R \rightrightarrows U$ as above we have $Ner(h(R \rightrightarrows U)) = U_{\bullet}$.
\end{lemma}

\begin{proof}
By definition we have $Ob(h(R \rightrightarrows U)(S)) = Hom_{Sm/k} (S, U)$ and $Mor(h(R \rightrightarrows U)(S)) = Hom_{Sm/k} (S, R)$. The set of two composable morphisms in $h(R \rightrightarrows U)(S)$ is $(R \times_U R) (S)$
where $R \times_U R$ is the fiber product of the maps $s : R \to U$ and $t : R \to U$. More generally the set of $n$-composable morphisms in $h(R \rightrightarrows U)(S)$ is 
$R \times_U R \times_U \dots \times_U R$ ($n$ times).
Since $R \cong U \times_F U$ and the maps $s$ and $t$ are first and second projections respectively we have
$$(Ner(h(R\rightrightarrows U)))_n (S) = \underbrace{(U \times_F U) \times_U  \cdots \times_U (U \times_F U)}_{n \, \text{times}}$$ 
which is isomorphic to
$U_n$.
\end{proof}

\begin{theorem} \label{spaces}
Let $U \to F$  be an atlas for a smooth 
Deligne-Mumford stack $F$. There is a canonical \'etale local weak equivalence $U_{\bullet} \to Sp(F)$.
\end{theorem}

\begin{proof}
We know that $h(R \rightrightarrows U)$ is locally weakly equivalent to $F$. Hence the morphism $Ner(h(R \rightrightarrows U)) \to Sp(F)$ is a local weak equivalence. The claim follows now 
from lemma \ref{mine}.
\end{proof}

\begin{corollary} 
\label{atlas motive}
Let $U \to F$ be an atlas for a smooth Deligne-Mumford stack $F$. The canonical map $M(U_{\bullet}) \to M(F)$ in
$\mathbf{DM}^{\eff} (k,\mathbb{Q})$ is an isomorphism.
(This is also true integrally.)
\end{corollary}

\begin{proof}
This follows from  proposition \ref{motivefunctor} and  theorem \ref{spaces}. 
\end{proof}

Let $F' \to F$ be a morphism of strict presheaves of groupoids. Let $F'_{\bullet}$ be the simplicial presheaf of groupoids such that
$F'_i := F' \times_F F' \times_F \dots \times_F F'$
($i+1$ times). Let $Ner(F'_{\bullet})$ be the bi-simplicial presheaf such that 
$Ner(F'_{\bullet})_{\bullet, i} := Ner(F'_i)$. Let $diag(Ner(F'_{\bullet}))$ be the diagonal.

\begin{lemma}
\label{etale cover}
Let $p : F' \to F$ be an \'etale, representable, surjective morphsim of Deligne-Mumford stacks (here stacks are strict presheaves of groupoids). Then the canonical morphism 
$$diag(Ner(F'_{\bullet})) \to Ner(F)$$ is an \'etale local weak equivalence. 
\end{lemma}

\begin{proof}
Let $U \to F$ be an atlas and $U_{\bullet}$ be the associated \u{C}ech simplicial scheme. Let $U'_{\bullet, \bullet}$ be the bi-simplicial algebraic space such that
$U'_{\bullet,i} := U_{\bullet} \times_F F'_i$ for $i \geq 0$. Hence, $U'_{j, \bullet} := U_j \times_F F'_{\bullet}$ for $j \geq 0$. There are natural morphisms 
$diag(U'_{\bullet, \bullet}) \to diag(Ner(F'_{\bullet}))$ and $diag(U'_{\bullet, \bullet}) \to U_{\bullet}$. For $i,j \geq 0$, $U'_{\bullet,i} \to F'_i$ and
$U'_{j, \bullet} \to U_j$ are \'etale \u{C}ech hypercovering, hence $U'_{\bullet,i} \to Ner(F'_i)$ and $U'_{j, \bullet} \to U_j$ are \'etale local weak equivalences.
By \cite[XII.3.3]{bk} $$diag(U'_{\bullet, \bullet}) \cong hocolim_{n \in \bigtriangleup} (U'_{\bullet, n}) \cong diag(Ner(F'_{\bullet}))$$ and
$$diag(U'_{\bullet, \bullet}) \cong hocolim_{n \in \bigtriangleup} (U'_{n, \bullet}) \cong U_{\bullet}.$$ 
This proves the lemma.
\end{proof}

\section{Motives of Deligne-Mumford Stacks, I}
\label{2}
In this section, we first show that the motive of a separated Deligne-Mumford stack is naturally isomorphic to the motive of its coarse moduli space. 
We also prove blow-up and projective bundle formulas
for smooth Deligne-Mumford stacks. We end the section with the construction of the Gysin triangle associated with a smooth
closed substack $Z$ of a smooth Deligne-Mumford stack $F$.

\subsection{Motive of coarse moduli space}
Let $F$ be a separated Deligne-Mumford stack. A coarse moduli space for $F$ is a map 
$\pi : F \to X$ to an algebraic space such that $\pi$ is initial among maps from $F$ to algebraic spaces, and for every
algebraically closed field $k$ the map $[F(k)] \to X(k)$ is bijective (where $[F(k)]$ denotes the
set of isomorphism classes of objects in the small category $F(k)$). If $F$ is a separated Deligne-Mumford stack over a 
field $k$ of characteristic $0$, then a
coarse moduli space $\pi : F \to X$ exists.

Let $X$ be a scheme and let $G$ be a group acting on $X$. Then $G$ acts on the presheaf $\mathbb{Q}_{tr}(X)$. Let $\mathbb{Q}_{tr}(X)_G$ be the $G$-coinvariant presheaf, such that for any $Y \in Sm/k$ we have
$\mathbb{Q}_{tr}(X)_G(Y) := (\mathbb{Q}_{tr}(X)(Y))_G$.

\begin{lemma}
\label{motive of quotient stacks}
Let $X$ be a smooth quasi-projective scheme and let $G$ be a finite group acting on $X$. Let $[X/G]$ be the quotient Deligne-Mumford stack 
and $X/G$ be the quotient scheme. Then
\begin{enumerate}
 \item $M([X/G]) \cong \mathbb{Q}_{tr}(X)_G$ in $\mathbf{DM}^{\eff}(k,\mathbb{Q})$;
 \item $\mathbb{Q}_{tr}(X/G) \cong \mathbb{Q}_{tr}(X)_G$ as presheaves.
\end{enumerate}
Hence the canonical morphism $M([X/G]) \to \mathbb{Q}_{tr}(X/G)$ is an isomorphism in $\mathbf{DM}^{\eff}(k,\mathbb{Q})$.  
\end{lemma}

\begin{proof}
To deduce (1), we observe that the morphism $X \to [X/G]$ sending $X$ to the trivial $G$-torsor $X \times G \to X$ is an \'etale atlas. Let $X_{\bullet}$ be the corresponding 
\u{C}ech simplicial scheme. Then 
$\mathbb{Q}_{tr}(X_{\bullet}) \cong M([X/G])$ in $\mathbf{DM}^{\eff}(k,\mathbb{Q})$. Moreover, 
$\mathbb{Q}_{tr}(X_{\bullet})(Y) \cong (\mathbb{Q}_{tr}(X)(Y)\otimes \mathbb{Q}[EG])/G$.
Hence the complex $\mathbb{Q}_{tr}(X_{\bullet})(Y)$ computes the homology of $G$ with coefficient in the $G$-module $\mathbb{Q}_{tr}(X_{\bullet})(Y)$. Since $G$ is finite and we work
with rational coefficients,
we have $ \mathbb{Q}_{tr}(X_{\bullet})(Y) \cong (\mathbb{Q}_{tr}(X)(Y))_G$ in the derived category of chain complexes of $\mathbb{Q}$-vector spaces.

To deduce (2), we observe that the canonical quotient morphism $\pi : X \to X/G$ is finite and surjective. Let $m$ be the generic degree of $\pi$, then
the morphism $\Gamma_{\pi} : \mathbb{Q}_{tr}(X) \to \mathbb{Q}_{tr}(X/G)$ has a section $\frac{1}{m} ^t\Gamma_{\pi}$. Hence $\mathbb{Q}_{tr}(X/G)$ is isomorphic to the image of the projector $\frac{1}{m} ^t\Gamma_{\pi} \circ \Gamma_{\pi}$. 
But  $\frac{1}{m} ^t\Gamma_{\pi} \circ \Gamma_{\pi} =  \frac{1}{|G|}\sum_{g \in G} g$ whose image is 
isomorphic to $\mathbb{Q}_{tr}(X)_G$. 
\end{proof}

\begin{remark}
In the proof above, the composition $\frac{1}{m} ^t\Gamma_{\pi} \circ \Gamma_{\pi}$ is well defined (cf.~\cite[Definition 1A.11]{mvw}). Indeed, since $X/G$ is normal, the finite correspondence
$^t\Gamma_{\pi}$ is a relative cycle over $X/G$ by \cite[Theorem 1A.6]{mvw}.
\end{remark}

\begin{theorem}
Let $F$ be a separated smooth Deligne-Mumford stack over a field $k$ of characteristic $0$. Let $\pi : F \to X$ be the coarse moduli space. Then the natural morphism
$M(\pi) : M(F) \to \mathbb{Q}_{tr}(X)$ is an isomorphism in $\mathbf{DM}^{\eff}(k,\mathbb{Q})$. 
\end{theorem}

\begin{proof}
By \cite[Prop 1.17]{t1} and \cite[Prop 2.8]{vi}, there exists an \'etale covering $(U_i)_{i \in I}$ of $X$, such that $U_i \cong X_i/H_i$ and 
$F_i := U_i \times_X F \cong [X_i / H_i]$ for quasi-projective smooth schemes $X_i$ and finite groups $H_i$. Let $F' := \coprod F_i$ and $X' := \coprod X_i/H_i$. 
Then by lemma \ref{etale cover},
$M(diag(Ner(F'_{\bullet}))) \cong M(F)$ in $\mathbf{DM}^{\eff}(k,\mathbb{Q})$.
Similarly, $\mathbb{Q}_{tr}(X'_{\bullet}) \cong \mathbb{Q}_{tr}(X)$. 
To show that $M(F) \cong \mathbb{Q}_{tr}(X)$, it is then enough to 
show that $M(Ner(F'_n)) \cong M(X'_n)$. Hence we are reduced to the case $F=[X/G]$ which follows from lemma \ref{motive of quotient stacks}.
\end{proof}

\subsection{Motive of a projective bundle}
Let $\mathcal{E}$ be a vector bundle of rank $n+1$ on a smooth finite type Deligne-Mumford stack $F$ and let $Proj(\mathcal{E})$ denote the associated projective bundle over $F$. 

\begin{theorem} \label{projective bundle} 
There exists a canonical isomorphism in $\mathbf{DM}^{\eff}(k,\mathbb{Q})$:
$$M(Proj(\mathcal{E})) \to \bigoplus_{i=0}^n  M(F) \otimes \mathbb{Q}(i)[2i].$$

\end{theorem}

\begin{proof}
Let $a : U \to F$ be an atlas of $F$ and $V:= Proj(a^* (\mathcal{E})) \to Proj(\mathcal{E})$ be the induced atlas of $Proj(\mathcal{E})$. 

The line bundle $\mathcal{O}_{Proj(\mathcal{E})}(1)$ induces a canonical map 
$$\tau:M(Proj(\mathcal{E})) \to \mathbb{Q}(1)[2]$$ 
in $\mathbf{DM}^{\eff}(k,\mathbb{Q})$ by corollary \ref{picard} below. Here we take 
$\mathbb{Q}(1)[2]:=C_*(\mathbb{P}^1, \infty)=N(\underline{Hom} (\bigtriangleup^{\bullet}, \mathbb{Q}_{tr}(\mathbb{P}^1, \infty)))$ where $N$ is the normalized chain complex and $\underline{Hom}$ is the internal $Hom$ (see 
\cite[page 15-16]{mvw} for $\mathbb{Q}_{tr}(\mathbb{P}^1, \infty)$ and $C_*$). As the complex $C_*(\mathbb{P}^1, \infty)$ is fibrant for the projective motivic model structure (see \cite[Corollary 2.155]{a2}), $\tau$ is represented by a morphism 
$$\tau':N(\mathbb{Q}_{tr} (V_{\bullet})) 
\to C_*(\mathbb{P}^1, \infty)$$ 
in $K(PST(k))$ where $V_{\bullet}$ is the \u{C}ech complex associated to the atlas $a : V \to Proj(\mathcal{E})$. 
By the Dold-Kan correspondence we get a morphism 
\begin{equation}
\label{tao}
\tau':\mathbb{Q}_{tr} (V_{\bullet}) \to \underline{Hom}(\bigtriangleup^{\bullet}, \mathbb{Q}_{tr}(\mathbb{P}^1, \infty))
\end{equation} 
in $\bigtriangleup^{op}(PST(k))$. 
Note that in simplicial degree zero, the induced map 
$\mathbb{Q}_{tr}(V) \to \mathbb{Q}_{tr}(\mathbb{P}^1, \infty)$ represents the class of 
$\mathcal{O}_{Proj(\mathcal{E}|_{V})}(1)$. 
Using the commutativity of
$$\xymatrix{\mathbb{Q}_{tr} (V_i) \ar[r] \ar[d] & 
\underline{Hom}(\bigtriangleup^{i},\mathbb{Q}_{tr}(\mathbb{P}^1, \infty)) \ar[d]^{we}\\
\mathbb{Q}_{tr}(V) \ar[r]& \mathbb{Q}_{tr}(\mathbb{P}^1,\infty)}$$
the upper horizontal morphism also represents the class of $\mathcal{O}_{Proj(\mathcal{E}|_{V_i})}(1)$ modulo the $\mathbb{A}^1$-weak equivalence $we$.

The morphism of simplicial presheaves \eqref{tao}  induces a morphism 
$$(\tau')^m : \mathbb{Q}_{tr} (\underbrace{V_{\bullet} \times \dots \times V_{\bullet}}_{m \, \text{times}}) \to \underline{Hom} ( \underbrace{\bigtriangleup^{\bullet} \times \dots \times \bigtriangleup^{\bullet}}_{m \, \text{times}}, \mathbb{Q}_{tr}(\mathbb{P}^1, \infty)^{\wedge m})$$
between multisimplicial presheaves with transfers for every positive integer $m$.
The diagonals $\bigtriangleup^{\bullet} \to diag(\bigtriangleup^{\bullet} \times \dots \times \bigtriangleup^{\bullet})$ and $ V_{\bullet} \to diag(V_{\bullet} \times \dots \times V_{\bullet})$ give a morphism 
$$(\tau' )^m : \mathbb{Q}_{tr}(V_{\bullet}) \to  \underline{Hom} ( \bigtriangleup^{\bullet}, \mathbb{Q}_{tr}(\mathbb{P}^1, \infty)^{\wedge m}).$$

Moreover the morphism $\mathbb{Q}_{tr} (V_{\bullet}) \to \mathbb{Q}_{tr} (U_{\bullet})$ gives a morphism
$$\sigma : \mathbb{Q}_{tr}(V_{\bullet}) \to diag(\bigoplus_{m=0}^n \underline{Hom} ( \bigtriangleup^{\bullet}, \mathbb{Q}_{tr}(\mathbb{P}^1, \infty)^{\wedge m} \otimes \mathbb{Q}_{tr}(U_{\bullet})))$$ of simplicial persheaves with transfers. Here $U_{\bullet}$ is the associated \u{C}ech complex of $a:U\to F$.

In  degree $i$ the morphism $\sigma$ coincides with the one from \cite[Construction 15.10]{mvw} modulo the $\mathbb{A}^1$-weak equivalence 
$$\bigoplus_{m=0}^n \underline{Hom}(\bigtriangleup^{i}, \mathbb{Q}_{tr}(\mathbb{P}^1, \infty)^{\wedge m} \otimes \mathbb{Q}_{tr}(U_i)) \to \bigoplus_{m=0}^n  \mathbb{Q}_{tr}(\mathbb{P}^1, \infty)^{\wedge m} \otimes \mathbb{Q}_{tr}(U_i).$$
It follows from \cite[Theorem 15.12]{mvw} that $\sigma$ induces $\mathbb{A}^1$-weak equivalence after passing to the normalized complex. This proves the theorem.
\end{proof}

\subsection{Motives of blow-ups} 
Let $X$ be a $k$-scheme. Let $X' \to X$ be a blow-up with center $Z$ and $Z' := Z \times_X X'$ be the exceptional divisor. Then \cite[Theorem 13.26]{mvw} can be rephrased as follows. (Recall that $char(k)=0$.)

\begin{theorem} \label{homotopy}
The following commutative diagram
$$\xymatrix{\mathbb{Z}_{tr} (Z') \ar[r] \ar[d] &
\mathbb{Z}_{tr} (X') \ar[d] \\
\mathbb{Z}_{tr} (Z) \ar[r]  &\mathbb{Z}_{tr} (X)}$$
is homotopy co-cartesian (with respect to the \'etale $\mathbb{A}^1$-local model structure).
\end{theorem}

\begin{proposition} \label{blowup-form}
Let $F$ be a smooth Deligne-Mumford stack and $Z \subset F$ be a smooth closed substack. Let $Bl_Z(F)$ be the blow-up of $F$ with center $Z$ and 
$E := Z \times_F Bl_Z(F)$ be the exceptional divisor. Then one has a canonical distinguished triangle of the form :
 $$ M(E) \to M(Z) \oplus M(Bl_Z(F)) \to M(F) \to M(E)[1].$$
\end{proposition}
\begin{proof}
 Let $a : U \to F$ be an atlas and let $U_{\bullet}$ be the associated \u{C}ech complex. Then the following square of simplicial presheaves with transfers

\[
\xymatrix{
         \mathbb{Q}_{tr} (U_{\bullet} \times_F E) \ar[r] \ar[d] &\mathbb{Q}_{tr} (U_\bullet \times_F Bl_Z(F)) \ar[d] \\
         \mathbb{Q}_{tr} (U_{\bullet} \times_F Z) \ar[r]  &\mathbb{Q}_{tr} (U_{\bullet}) 
                                    }
\]
is homotopy co-cartesian in each degree by theorem \ref{homotopy}. Since homotopy colimits commutes with homotopy push-outs, the following square 
\[
\xymatrix{
         M(E) \ar[r] \ar[d] &M(Bl_Z(F)) \ar[d] \\
         M(Z) \ar[r]  &M(F) 
                                    }
\]
is homotopy co-cartesian and hence we get our result.
\end{proof}

\begin{theorem} 
\label{blowupform}
Let $F$ be a smooth Deligne-Mumford stack and $Z \subset F$ be a smooth closed substack of pure codimension $c$. Let $Bl_Z(F)$ be the blow-up of $F$ with center $Z$. Then
$$M(Bl_Z(F)) \cong M(F) \;\bigoplus\; (\oplus_{ i=1}^{c-1} M(Z)(i)[2i]).$$ 
\end{theorem}

\begin{proof}
By proposition \ref{blowup-form} we have a canonical distinguished triangle
$$ M(p^{-1}(Z)) \to M(Z) \oplus M(Bl_Z(F)) \to M(F) \to M(p^{-1}(Z))[1],$$
where $p : Bl_Z(F) \to F$ is the blow-up. 
Since $Z$ is smooth, $p^{-1}(Z) \cong Proj(N_Z(F))$, where $N_Z(F)$ is the normal bundle. Hence using theorem \ref{projective bundle}, it is enough to show that the morphism $M(F) \to M(p^{-1}(Z))[1]$ is zero in $\mathbf{DM}^{\eff}(k,\mathbb{Q})$.
Let $q : Bl_{Z \times \left\{0\right\}} (F \times \mathbb{A}^1) \to F \times \mathbb{A}^1$ be the 
blow-up of $Z \times \left\{0\right\}$ in $F \times \mathbb{A}^1$.

Following the  proof of \cite[Proposition 3.5.3]{voe}, 
consider the morphism of exact triangles:
$$\xymatrix{
M(p^{-1}(Z)) \ar[r] \ar[d] & M(q^{-1}(Z \times \left\{0\right\}) \ar[d] \\
M(Z) \oplus M(Bl_Z F) \ar[r] \ar[d]   & M(Z \times \left\{0\right\}) \oplus M(Bl_{Z \times \left\{0\right\}} F \times \mathbb{A}^1) \ar[d]^f \\
M(F) \ar[r]^{s_0} \ar[d]^g &M(F \times \mathbb{A}^1) \ar[d]^h \\
M(p^{-1} (Z))[1] \ar[r]^{a} & M(q^{-1}(Z \times \left\{0\right\}))[1]
}$$

Since the morphism $s_0$ is an isomorphism and since by theorem \ref{projective bundle} $a$ is split injective, 
the morphism $g$ is zero if $h$ is zero.
To show that $h$ is zero it is enough to show that $f$ has a section.  
This is the case as the composition $$M(F \times \left\{1\right\}) \to M(Bl_{Z \times \left\{0\right\}} F \times \mathbb{A}^1) \to M(F \times \mathbb{A}^1)$$
is an isomorphism.
\end{proof}

\subsection{Gysin triangle}

Given a morphism $F \to F'$ of Deligne-Mumford stacks, let $$M\left(\frac{F'}{F} \right):= cone(M(F) \to M(F')).$$ 
Similarly given a morphism
$V_{\bullet} \to U_{\bullet}$ of simplicial schemes, 
let $$\mathbb{Q}_{tr}\left(\frac{U_{\bullet}}{V_{\bullet}}\right) := cone(\mathbb{Q}_{tr}(V_{\bullet}) \to \mathbb{Q}_{tr}(U_{\bullet})).$$

\begin{lemma} 
\label{excision}
Let $f : F' \to F$ be an \'etale morphism of smooth Deligne-Mumford stacks, 
and let $Z \subset F$ be a closed substack such that $f$ induces an isomorphism $f^{-1} (Z)\cong Z$. 
Then the canonical morphism
$$M\left(\frac{F'}{F' - Z}\right) \to M\left(\frac{F}{F - Z}\right)$$ 
is an isomorphism.
\end{lemma}

\begin{proof}
Let $v' : V' \to F'$ be an atlas of $F'$, and let $v : V \to F - Z$ be an atlas of the complement of $Z$. Then 
$U=V \coprod V' \to  F$ is an atlas
of $F$. Let $f_{\bullet} : V'_{\bullet} \to U_{\bullet}$ be the induced morphism between the associated \u{C}ech simplicial schemes. 
In each simplicial degree $i$, 
we have an \'etale morphism $f_i : V'_i \to U_i$ such that $f_i$ induces an isomorphism 
$Z \times_F V'_i \cong Z \times_F U_i$. Let $Z_{\bullet} := Z \times_F U_{\bullet}\cong Z \times_F V'_{\bullet}$.
It is enough to show that the canonical morphism 
$M\left(\frac{V_{\bullet}}{V_{\bullet}- Z_{\bullet}}\right) \to M\left(\frac{U_{\bullet}}{U_{\bullet}- Z_{\bullet}}\right)$ is an isomorphism. This is indeed the case as $\mathbb{Q}_{tr} \left(\frac{V_{\bullet}}{V_{\bullet}- Z_{\bullet}}\right) \cong \mathbb{Q}_{tr} \left(\frac{U_{\bullet}}{U_{\bullet}- Z_{\bullet}}\right)$ by \cite[Proposition 5.18]{voe1}.
\end{proof}

%\begin{remark}
%\label{Mayer Vietoris}
%Let $F$ be a smooth Deligne-Mumford stack 
%and let $F = F_1 \cup F_2$, where $F_i$'s are open substacks of $F$. 
%We have an exact triangle:
%$$ M(F_1 \times_F F_2) \to M(F_1) \oplus M(F_2) \to M(F) \to M(F_1 \times_F F_2)[1].$$
%This follows from lemma \ref{excision} applied to 
%$F'=F_1$ and $Z =F\setminus F_2$.
%\end{remark}

\begin{lemma}
\label{vector bundle}
Let $p:V \to F$ be a vector bundle of rank $d$ over a smooth 
Deligne-Mumford stack $F$. Let $s:F \to V$ be the zero 
section of $p$. Then 
$$M\left(\frac{V}{V \setminus s} \right) \cong M(F)(d)[2d].$$  
\end{lemma}

\begin{proof}
Using lemma \ref{excision}, we have an isomorphism 
$$M\left(\frac{V}{V \setminus s}\right) \cong M\left(\frac{Proj(V \oplus O)}{Proj(V \oplus O) \setminus s}\right).$$ 
The image of the embedding $Proj(V) \to Proj(V \oplus O)$ is disjoint from $s$ and 
$\iota : Proj(V) \to Proj(V \oplus O)\setminus s$ 
is the zero section of a line bundle. 
Thus, the induced morphism $\iota:
M(Proj(V)) \to M(Proj(V \oplus O)\setminus s)$
is an $\mathbb{A}^1$-weak equivalence. (This can be checked 
using an explicit $\mathbb{A}^1$-homotopy as in the classical case 
where the base is a scheme.)
It follows that 
$$M\left(\frac{V}{V \setminus s}\right) \cong
M\left(\frac{Proj(V \oplus O)}{Proj(V)}\right).$$
Now using theorem \ref{projective bundle}, we get
$$M\left(\frac{Proj(V \oplus O)}{Proj(V)}\right) \cong M(F)(d)[2d].$$
This proves the lemma.
\end{proof}

\begin{theorem} 
\label{gysin triangle}
Let $Z \subset F$ be a smooth closed codimension $c$ substack of a smooth Deligne-Mumford stack $F$. 
Then there exists a Gysin exact triangle:
$$ M(F \setminus Z) \to M(F) \to M(Z)(c)[2c] \to M(F \setminus Z)[1].$$
\end{theorem}

\begin{proof}
We have the following obvious exact triangle
$$ M(F \setminus Z) \xrightarrow{i} M(F)  \to  M\left(\frac{F}{F \setminus Z}\right) \to M(F \setminus Z)[1].$$ 
We need to show that 
$M\left(\frac{F}{F \setminus Z}\right) \cong M(Z)(c)[2c]$ 
in $\mathbf{DM}^{\eff}(k, \mathbb{Q})$. Let $D_Z(F)$ be the space of deformation to the normal cone and let $N_Z(F)$ be the normal bundle. Consider the following commutative diagram of stacks:
$$\xymatrix{ 
Z \times 1 \ar[r] \ar@{^{(}->}[d] &Z \times \mathbb{A}^1 \ar@{^{(}->}[d] &Z \times 0 \ar[l]\ar@{^{(}->}[d]^{s_Z} \\
F \times 1 \ar@{^{(}->}[r] \ar[d] &D_Z(F) \ar[d] &N_Z(F) \ar@{_{(}->}[l]\ar[d] \\
1 \ar[r] &\mathbb{A}^1  &0. \ar[l] 
}$$
This gives morphisms 
$$s^1 : M\left(\frac{F}{F \setminus Z} \right) \to M\left(\frac{D_Z(F)}{D_Z(F) \setminus (Z \times \mathbb{A}^1)} \right)$$ and 
$$s^0 : M\left(\frac{N_Z(F)}{N_Z(F) \setminus s_Z} \right) \to M\left(\frac{D_Z(F)}{D_Z(F) \setminus (Z \times \mathbb{A}^1)} \right).$$ 
Let $U \to F$ be an atlas of $F$ and let $U_{\bullet}$ be the associated \u{C}ech simplicial scheme. Then $s^1$ can be described as  
$$s^1 :  \mathbb{Q}_{tr}\left(\frac{U_{\bullet}}{(F \setminus Z) \times_F U_{\bullet}} \right) \to \mathbb{Q}_{tr}\left(\frac{D_Z(F) \times_F U_{\bullet}}{(D_Z(F) \setminus (Z \times \mathbb{A}^1))\times_F U_{\bullet}} \right).$$
Let $Z_i := Z \times_F U_i$.
In each simplicial degree $i$ the morphism 
$(s^1)_i : \mathbb{Q}_{tr}\left(\frac{U_i}{U_i \setminus (Z_i)} \right) \to \mathbb{Q}_{tr}\left(\frac{D_{Z_i}(U_i)}{D_{Z_i}(U_i) \setminus (Z_i \times \mathbb{A}^1)}\right)$ induced
by $s^1$ is an $\mathbb{A}^1$-weak equivalence by lemma \ref{codimension one}. Hence $s^1$ is an $\mathbb{A}^1$-weak equivalence. Similarly,
$s^0$ is an $\mathbb{A}^1$-weak equivalence. Hence we get an isomorphism $M\left(\frac{F}{F \setminus Z} \right) \cong M\left(\frac{N_Z(F)}{N_Z(F) \setminus s_Z} \right)$.
But $M\left(\frac{N_Z(F)}{N_Z(F) \setminus s_Z} \right) \cong M(Z)(c)[2c]$ by lemma \ref{vector bundle}. 
\end{proof}

\begin{lemma} 
\label{codimension one}
Suppose we have a cartesian diagram of smooth schemes
$$\xymatrix{ 
 Z \ar[d]^0\ar[r] & Y  \ar[d]^-u \\
\mathbb{A}^{1} \times Z \ar[r]^-v & X  
}$$
where $u$ and $v$ are closed embeddings and 
$Y$ has codimension $1$ in $X$.
Then the canonical
morphism
$M\left(\frac{Y}{Y\setminus Z}\right) \to M\left(\frac{X}{X \setminus (\mathbb{A}^1 \times Z)}\right)$ is an isomorphism.
\end{lemma}

\begin{proof}
Let $c$ be the codimension of $Z$ in $Y$; it is also the codimension of $\mathbb{A}^1\times Z$ in $X$. Using  \cite[Proposition 3.5.4]{voe} we have $M\left(\frac{Y}{Y\setminus Z}\right) \cong M(Z)(c)[2c]$ and 
$M\left(\frac{X}{X \setminus (\mathbb{A}^1 \times Z)}\right) \cong M(Z \times \mathbb{A}^1)(c)[2c]$. Since $M(Z) \cong M(Z \times \mathbb{A}^1)$, we get the lemma.
\end{proof}

\section{Motives of Deligne-Mumford stacks, II} 
\label{3}

The main goal of this section is to show 
that the motive of a smooth 
Deligne-Mumford stack $F$ is a direct factor of 
the motive of a smooth and quasi-projective variety. 
Moreover, if $F$ is proper, this variety can be chosen to 
be projective.

\subsection{Blowing-up Deligne-Mumford stacks and principalization}

Let $F$ be a smooth Deligne-Mumford stack and let $a : U \to F$ be an atlas of $F$. Let $Z$ be a closed substack of $F$. The blow-up of $F$ along $Z$ is a Deligne-Mumford stack
$Bl_Z F$ together with a representable projective morphism $\pi : Bl_Z F \to F$. The induced morphism $a': Bl_{Z \times_F U} U \to Bl_Z F$ is an atlas.  The existence of $Bl_Z F$ is a consequence of the fact that blow-ups commute with flat base change.

\begin{theorem} \label{resolution}
Let $F$ be a smooth Deligne-Mumford stack of finite type over a field of characteristic zero. Let $\mathcal{O}_F$ be the structure sheaf and let $\mathcal{I} \subset \mathcal{O}_F$ be a coherent ideal. Then there is a sequence of blow-ups in smooth centers 
$$\pi : F_r \xrightarrow{\pi_{r}} F_{r-1} \xrightarrow{\pi_{r-1}} \dots \xrightarrow{\pi_1} F   $$
such that $\pi^* \mathcal{I} \subset \mathcal{O}_{F_r}$ is locally principal. 

\end{theorem}

\begin{proof}
Let $a : U \to F$ be an atlas and denote $\mathcal{J} := a^*\mathcal{I}$. By Hironaka's resolution of singularities  \cite[Theorem 3.15]{kol}, we have a sequence of blow-ups  in smooth centers
$$\pi' : U_r \xrightarrow{\pi'_{r}} U_{r-1} \xrightarrow{\pi'_{r-1}} \dots \xrightarrow{\pi'_1} U   ,$$ 
such that  
$(\pi')^*\mathcal{J}$ is a locally principal coherent ideal on $U_r$. Moreover this sequence commutes with arbitary smooth base change. Hence the sequence $\pi'$ 
descends  to give the sequence of the statement.
\end{proof}

\begin{lemma} 
\label{coarse moduli space}
Let $F' \to F$ be a (quasi-)projective representable morphism of Deligne-Mumford stacks. Let $X$ and $X'$ be the coarse moduli spaces of $F$ and $F'$ respectively.  Then the induced morphism $X' \to X$ is (quasi-)projective. In particular, if $X$ is (quasi-)projective then so is $X'.$
\end{lemma}

\begin{proof}
\cite[lemma 2, Theorem 1]{kv}.
\end{proof}

The proof of the following theorem was communicated to us by David Rydh.
\begin{theorem} \label{quotient}
Given a smooth finite type Deligne-Mumford stack $F$ over $k$, there exists a sequence of blow-ups in smooth centers $\pi : F' \to F$, such that the coarse moduli space of $F'$ is quasi-projective. 
\end{theorem}

\begin{proof}

 Let $p : F \to X$ be the morphism to the coarse moduli space of $F$. $X$ is a separated algebraic space. By Chow's Lemma (\cite[Theorem 3.1]{kn}) we have a projective morphism $g :X' \to X$ from a quasi-projective scheme $X'$. Moreover by \cite[Corollary 5.7.14]{rg} we may assume that $g$ is a blow-up along a closed subspace $Z \subset X$.
Let $F' := X' \times_X F$. There is a morphism $p' : F' \to X'$. Since $F$ is tame  $X'$ is the coarse moduli space of $F'$ (\cite[Cor 3.3]{aov}). 
Let $T := Z \times_X F$ and let $\pi : Bl_{T} F \to F$ be the blow-up of $F$ along $T$.  Then $Bl_{T}F$ is the closure of $F \setminus T$ in $F'$.
As $F'$ is tame the coarse moduli space of $Bl_{T} F$ is a closed subscheme of $X'$. Hence $Bl_{T} F$ has quasi-projective coarse moduli space.

Now by  \ref{resolution} We have a sequence of blow-ups in smooth centers $\pi : F_r \to F$ such that the ideal sheaf defining $T$ is principalized. Hence there exists a canonical projective representable morphism $\pi' : F_r \to Bl_{T} F$. 
Since  $Bl_{T} F$ has quasi-projective coarse moduli space and $\pi'$ is a projective representable morphism, we have our result by lemma \ref{coarse moduli space}.
\end{proof}

\subsection{Chow motives and motives of proper Deligne-Mumford stack}

Let $f : X \to Y$ be a finite morphism between smooth schemes such that each connected component of $X$ maps surjectively to a connected component of $Y$ and generically over $Y$ the degree of $f$ is constant equal to $m$. Then the transpose of $\Gamma_f$ is a correspondence from $Y$ to $X$. This defines a morphism $^tf : \mathbb{Q}_{tr} (Y) \to \mathbb{Q}_{tr}(X)$ such that $ f \circ (\frac{1}{m} ^tf)$ is the identity.

\begin{remark} \label{inverse}
Suppose we are given a cartesian diagram of smooth schemes
\[
\xymatrix{
         Y' \ar[r]^{f'} \ar[d]^{h'} &X' \ar[d]^{h} \\
         Y \ar[r]^f  &X 
                                    }
\]
with $h$  \'etale. Assume that $f$ is a finite morphism such that each connected component of $Y$ maps surjectively to a connected component of $X$ and generically over $X$ the degree of $f$ is constant equal to $m$. Then $f'$ satisfies the same properties as $f$. Thus we have morphisms $^tf :\mathbb{Q}_{tr}(X) \to \mathbb{Q}_{tr} (Y)$ and $^tf' : \mathbb{Q}_{tr}(X') \to \mathbb{Q}_{tr}(Y')$. Using the definition of composition of finite correspondences one can easily verify 
that $(h') \circ (^tf') = (^tf) \circ (h)$.
\end{remark}

\begin{lemma} \label{finite}
Let $F$ be a smooth Deligne-Mumford stack. Assume that there exists a smooth scheme $X$ and a finite surjective morphism $g : X \to F$. Then 
$M(F)$ is a direct factor of $M(X)$.
\end{lemma}

\begin{proof}
We may assume that $F$ and $X$ are connected.
Let  $a: U \to F$ be an atlas and  $U_{\bullet}$ the associated \u{C}ech complex. Set $V_{\bullet} := U_{\bullet} \times_F X$.  Then  $g'_{\bullet} :  V_{\bullet} \to U_{\bullet}$ is finite and surjective of constant degree $m$ in each simplicial degree. It follows from \ref{inverse} that $^tg'_{\bullet} : \mathbb{Q}_{tr} (U_{\bullet}) \to \mathbb{Q}_{tr} (V_{\bullet})$ is a morphism of simplicial sheaves with transfers such that $g'_{\bullet} \circ (\frac{1}{m} ^tg'_{\bullet}) = id$. Hence $\mathbb{Q}_{tr} (U_{\bullet})$ is a direct factor of $\mathbb{Q}_{tr} (V_{\bullet})$.

Since  $V_{\bullet}$ is a \u{C}ech resolution of $X$, we have $\mathbb{Q}_{tr} (V_{\bullet}) \cong \mathbb{Q}_{tr} (X)$ by \cite[Proposition 6.12]{mvw}. This proves the result.
\end{proof}

\begin{theorem} \label{Chow}
Let $F$ be a proper (resp. not necessarily proper) smooth Deligne-Mumford stack. Then $M(F)$ is a direct summand of the motive of a projective (resp. quasi-projective) variety.
\end{theorem}
 
\begin{proof}
We can assume that $F$ is connected. By \ref{quotient}  we get a sequence of blow-ups with smooth centers $ \pi : F' \to F$ such that $F'$ has (quasi)-projective coarse moduli space. By \ref{blowupform} $M(F)$ is a direct summand of $M(F')$.
By \cite[Theorem 1]{kv}  there exists a smooth (quasi)-projective variety $X$ and a finite flat morphism $g : X \to F'$. Hence $M(F')$ is a direct summand of $M(X)$ which proves our claim.
\end{proof}

Recall that the category of effective geometric motives
$\mathbf{DM}_{gm}^{eff}(k, \mathbb{Q})$ is the thick subcategory of $\mathbf{DM}^{\eff} (k, \mathbb{Q})$ generated by the motives $M(X)$ for $X \in Sm/k$ (see \cite[Definition 14.1]{mvw}). 

\begin{corollary}
For any smooth finite type Deligne-Mumford stack $F$, $M(F)$ is an effective geometric motive.
\end{corollary}

\begin{remark}
\label{Chow embedding}
By \cite[Proposition 20.1]{mvw} the category of effective Chow motives embeds 
into $\mathbf{DM}^{\eff} (k, \mathbb{Q})$. Theorem \ref{Chow} shows that $M(F)$ lies in the essential image of this embedding for any smooth proper Deligne-Mumford stack $F$.

\end{remark}

\section{Motivic cohomology of stacks} \label{4}

Let $F$ be a smooth Deligne-Mumford stack. For each integer $i$ let $\mathbb{Q}(i) \in \mathbf{DM}^{\eff}(k, \mathbb{Q}) $ denote the motivic complex of weight $i$ with rational coefficients (see \cite[Definition 3.1]{mvw}). 

\begin{definition}
 The \'etale site $F_{\et}$ is defined as follows. The objects of $F_{\et}$  are couples $(X,f)$ with $X$ a scheme and $f : X \to F$ a representable \'etale morphism. A morphism from $(X,f)$ to $(Y,g)$ is a couple
$(\phi, \alpha)$, where $\phi:X \to Y$ is a morphism of schemes and 
$\alpha : f \cong g \circ \phi$ is a $2$-isomorphism. Covering families of an object 
$(U, u)$ are defined as families $\left\{u_i:U_i \rightarrow U \right\}_{i \in I}$ such that the $u_i$'s are \'etale and $\cup{u}_i : \coprod_i U_i \to U$ is surjective.
\end{definition}

\begin{definition}
The motivic cohomology of $F$ with rational coefficients is defined as $H_M^{2i-n} (F,i) := H^{2i-n} (F_{\et}, \mathbb{Q}(i)|_{F_{\et}})$.
\end{definition}

\begin{remark}
In \cite[3.0.2]{j2}, motivic cohomology of an algebraic stack $F$ is defined using the smooth site of $F$.  For Deligne-Mumford stacks, this coincides with our definition by \cite[Proposition 3.6.1(ii)]{j2}. 
\end{remark}

\begin{lemma} \label{motivic cohomology}
Let $F$ be a Deligne-Mumford stack. We have an isomorphism 
$$H_M^{2i-n} (F,i) \simeq Hom_{\mathbf{DM}^{\eff}(k, \mathbb{Q})} (M(F),\mathbb{Q}(n)[2i-n]).$$
\end{lemma}

\begin{proof}
Let $U \to F$ be an atlas and $U_{\bullet}$ be the associated \u{C}ech complex.  We have an \'etale weak equivalence $\mathbb{Q}(U_{\bullet}) \to \mathbb{Q}$ of complexes of sheaves on $F_{\et}$. Here $\mathbb{Q}$ is the constant sheaf on $F_{\et}$. 
Writing $D(F_{\et})$ for the derived category of sheaves of $\mathbb{Q}$-vector spaces on $F_{\et}$, we thus have 
$$H_M^{2i-n} (F,i) \cong  Hom_{D(F_{\underline{\et}})}(\mathbb{Q}(U_{\bullet}), \mathbb{Q}(i)[2i-n]).$$  Let $a : \mathbb{Q}(i) \to L$ be a fibrant replacement for the injective local model structure on $K(PST(k))$ and let $b: L|_{F_{\et}} \to M$ be a fibrant replacement for the injective local model structure on $K(F_{\et})$. 

Since both $a$ and $b$ are \'etale local weak equivalences the composition $b\circ a : \mathbb{Q}(i)|_{F_{\et}}
 \to M$ is an \'etale weak equivalence. It follows that 
$$Hom_{D(F_{\et})}(\mathbb{Q}(U_{\bullet}), (\mathbb{Q}(i)|_{F_{\et}})[2i-n]) \cong Hom_{Ho(K(F_{\et}))}(\mathbb{Q}(U_{\bullet}), M[2i-n]).$$  Using \cite[2.7.5]{wei}, it follows that $ H_M^{2i-n} (F,i)$ is the
$(2i-n)$-th cohomology of the complex $Tot(Hom(\mathbb{Q}(U_{\bullet}), M))$.

On the other hand, since $a$ is an \'etale weak equivalence and $\mathbb{Q}(i)$ is 
$\mathbb{A}^1$-local, $L$ is also $\mathbb{A}^1$-local. It follows that 
$$Hom_{\mathbf{DM}^{\eff}(k, \mathbb{Q})} (\mathbb{Q}_{tr}(U_{\bullet}), \mathbb{Q}(i)[2i-n]) \cong Hom_{Ho(K(PST(k)))} (\mathbb{Q}_{tr}(U_{\bullet}), L[2i-n]).$$ Again by (\cite[2.7.5]{wei}), the right hand side is same as $(2i-n)$-th cohomology of the complex  $Tot(Hom(\mathbb{Q}_{tr}(U_{\bullet}), L))$. 

To prove the lemma it is now sufficient to show that 
$L(X) \to M(X)$ is a quasi-isomorphism for any smooth $k$-scheme $X$. 
By definition 
$$H^n(L(X)) \cong H^n(Hom(\mathbb{Q}_{tr}(X), F)) \cong \mathbb{E}xt^n(\mathbb{Q}_{tr}(X), \mathbb{Q}(i)[n]) $$ and by \cite[6.25]{mvw}
we have $$\mathbb{E}xt^n(\mathbb{Q}_{tr}(X), \mathbb{Q}(i)[n]) = H^n_{\et}(X, \mathbb{Q}(i))$$ which is same as $H^n(M(X))$.
\end{proof}

\begin{corollary} \label{picard}
Let $F$ be a Deligne-Mumford stack and let $\mathcal{O}_F$ be the structure sheaf. Then we have an isomorphism
$$Pic(F) \otimes \mathbb{Q} \cong H^1_{\et}(F, \mathcal{O}_F^{\times} \otimes \mathbb{Q}) \cong Hom_{\mathbf{DM}^{\eff} (k, \mathbb{Q})} (M(F), \mathbb{Q}(1)[2])$$
\end{corollary}
\begin{proof}
The first isomorphism follows from \cite[page 65, 67]{mum}.   By \cite[Theorem 4.1]{mvw} $\mathcal{O}_F^*[1] \otimes \mathbb{Q} \cong \mathbb{Q}(1)[2]$. So the second isomorphism is a particular case of lemma \ref{motivic cohomology}.

\end{proof}

\begin{remark}
From the proofs, it is easy to see that 
Lemma \ref{motivic cohomology} and Corollary 
\ref{picard} are true integrally if we use
Voevodsky's category of \'etale motives with integral 
coefficients
$\mathbf{DM}^{\eff,\,\et}(k,\mathbb{Z})$.
\end{remark}

\section{Chow motives of stacks and comparisons} 
\label{5}

Let $F$ be a smooth Deligne-Mumford stack.

\begin{definition}[\cite{j2}] 
The codimension $m$ rational Chow group of $F$ is defined to be
$$A^m(F) := H_M^{2m} (F,m)_{\mathbb{Q}}.$$
\end{definition}

\begin{remark}
In \cite{g,t}, the rational Chow groups are defined as the \'etale cohomology of suitable $K$-theory sheaves. This agrees with our definition by \cite[Theorem 3.1, 5.3.10]{j2}.
\end{remark}
Let  $\mathcal{M}_k$ (resp. $\mathcal{M}_k^{\eff}$\,) be the category of covariant Chow motives (resp. effective Chow motives) with rational coefficients.
The construction of the category of Chow motives for smooth and proper Deligne-Mumford stacks using the theory $A^*$ was done in \cite[\S 8]{bm}.  
We will denote the category of Chow motives (resp. effective Chow motives) for smooth and proper Deligne-Mumford stacks by $\mathcal{M}_k^{DM}$ (resp. $\mathcal{M}_k^{DM,\,\eff}$\,).

Let $C$ be a symmetric monoidal category and let $X \in Ob(C)$. Recall that an object $Y \in C$ is called a strong dual of $X$ if there exist two morphisms 
$coev: \mathds{1} \to Y \otimes X$ 
and 
$ev : X \otimes Y \to \mathds{1}$, 
such that the composition of 
\begin{equation}
\label{eq:dual-1}
X \xrightarrow{id \otimes coev } X \otimes Y \otimes X \xrightarrow{ev \otimes id} X
\end{equation}
and the composition of
\begin{equation}
\label{eq:dual-2}
Y \xrightarrow{coev \otimes  id} Y \otimes X \otimes Y \xrightarrow{id \otimes ev} Y
\end{equation}
are identities.

\begin{lemma}
\label{dual}
Let $F$ be a proper smooth Deligne-Mumford stack of 
pure dimension $d$. 
Then $h_{DM}(F) := (F, \Delta_F, 0)$ has a strong dual 
in $\mathcal{M}_k^{DM}$. It is given by
$(F,\Delta_F,-d)$.
\end{lemma}

\begin{proof}
Set $h_{DM}(F)^*:=(F,\Delta_F,-d)$. We need to give morphisms $coev:\mathds{1} \to h_{DM}(F)^* \otimes h_{DM}(F)$ and 
$ev:h_{DM}(F) \otimes h_{DM}(F)^* \to \mathds{1}$, such that \eqref{eq:dual-1} and \eqref{eq:dual-2} are satisfied.
The morphisms $coev$ and $ev$ are given by $\Delta_F \in A^d(F \times F)$. To compute the composition of
\eqref{eq:dual-1}, we observe that
intersection of the cycles $\Delta_F \times \Delta_F \times F$ and $F \times \Delta_F \times \Delta_F$ in $F \times F \times F \times F \times F$ is equal to $\delta(F)$ where $\delta:F \to F \times F \times F \times F \times F$ is the diagonal morphism.
The push-forward to $F\times F$ of the latter is simply the diagonal 
of $F\times F$. This shows that the composition of
\eqref{eq:dual-1} is the identity of $h_{DM}(F)$. 
The composition of 
\eqref{eq:dual-2} is treated using the same method.
\end{proof}

By \cite[Theorem 2.1]{t} the natural functor 
$e: \mathcal{M}_k \to \mathcal{M}_k^{DM}$ is an equivalence of $\mathbb{Q}$-linear tensor categories. 
This equivalence preserves the subcategories of effective motives.
Thus, after inverting this equivalence we can associate an effective Chow motive 
$h(F) \in \mathcal{M}^{\eff}_k$ to every smooth and proper Deligne-Mumford stack $F$. 
On the other hand, 
by \cite[Proposition 20.1]{mvw} there exists a fully faithful functor $\iota:\mathcal{M}^{\eff}_k \to \mathbf{DM}^{\eff}(k, \mathbb{Q})$.

\begin{theorem}
\label{thm:compare-toen}
Let $F$ be a smooth proper Deligne-Mumford stack. Then $M(F)  \cong \iota \circ h(F)$.

\end{theorem}

\begin{proof}
We may assume that $F$ has pure dimension $d$.  By \ref{Chow} $M(F)$ 
is a direct factor of the motive of a smooth and projective variety $W$ such that $dim(W) =d$. 
By \cite[Example 20.11]{mvw},
$$\underline{Hom} (M(W), \mathbb{Q}(d)[2d]) \cong M(W)$$ 
is an effective Chow motive.
It follows that 
$\underline{Hom} (M(F), \mathbb{Q}(d)[2d])$ is also an effective Chow motive. 

We first show that $\iota \circ h(F) \cong \underline{Hom} (M(F), \mathbb{Q}(d)[2d])$. Let $\mathcal{V}_k$ be the category
of smooth and projective varieties over $k$.
For $M \in \mathbf{DM}^{\eff}(k, \mathbb{Q})$ denote $\omega_{M}$
the presheaf on $\mathcal{V}_k$ defined by 
$$X \in \mathcal{V}_k \mapsto  Hom_{\mathbf{DM}^{\eff}(k, \mathbb{Q})}(M(X), M).$$
Using \ref{fully faithful}, it is enough to construct 
an isomorphism of presheaves 
$$\omega_{\underline{Hom}(M(F), \mathbb{Q}(d)[2d])}\cong \omega_{\iota \circ h(F)}.$$
The right hand side is by definition the presheaf $A^{dim(F)} (- \times F)$. For $X \in \mathcal{V}_k$, we have 
$$\begin{array}{rcl}
\omega_{\underline{Hom}(M(F),\mathbb{Q}(d)[2d])}(X)
& = & \hom_{\mathbf{DM}^{\eff}(k,\mathbb{Q})}(M(X),
\underline{Hom}(M(F),\mathbb{Q}(d)[2d]))\\
& =& \hom_{\mathbf{DM}^{\eff}(k,\mathbb{Q})}(M(X\times F),\mathbb{Q}(d)[2d]))\\
& =& H_M^{2d} (X\times F ,d).
\end{array}$$ 
We conclude using
\cite[Theorem 3.1(i) and Theorem 5.3.10]{j2}.

To finish the proof, it remains to construct an isomorphism 
$\iota\circ h(F)\simeq \underline{Hom}(\iota\circ h(F),\mathbb{Q}(d)[2d])$. It suffices to do so in the stable triangulated category
of Voevodsky's motives $\mathbf{DM}(k,\mathbb{Q})$ in which 
$\mathbf{DM}^{\eff}(k,\mathbb{Q})$ embeds fully faithfully by 
Voevodsky's cancellation theorem.
(Recall that $\mathbf{DM}(k,\mathbb{Q})$ 
is defined as the homotopy category 
of $T=\mathbb{Q}_{tr}(\Aff^1/\Aff^1-0)$-spectra for 
the stable motivic model structure; 
for more details, see 
\cite[D\'efinition 2.5.27]{a1} in the special case where 
the valuation on $k$ is trivial.)
In $\mathbf{DM}(k,\mathbb{Q})$, 
we have an isomorphism
$$\underline{Hom}(\iota\circ h(F),\mathbb{Q}(d)[2d])
\simeq 
\underline{Hom}(\iota\circ h(F),\mathbb{Q}(0))\otimes \mathbb{Q}(d)[2d].$$
As the full embedding
$\mathcal{M}_k\to \mathbf{DM}(k,\mathbb{Q})$ and the 
equivalence $\mathcal{M}_k\simeq \mathcal{M}_k^{DM}$ are
tensorial, they preserve strong duals. From
Lemma \ref{dual}, it follows that 
$\underline{Hom}(\iota\circ h(F),\mathbb{Q}(0))$
is canonically isomorphic to 
$\iota(F,\Delta_F,-d)=\iota\circ h(F)\otimes \mathbb{Q}(-d)[-2d]$.
This gives the isomorphism 
$\underline{Hom}(\iota\circ h(F),\mathbb{Q}(d)[2d])
\simeq \iota\circ h(F)$ we want.
\end{proof}

\appendix

\section{}

\label{appendix-A}

As usual, we fix a base field $k$ of characteristic $0$. (Varieties will be always defined over $k$.) Recall that 
$\mathcal{M}^{\eff}_k$ is the category of effective Chow motives with rational coefficients. 
We will have to consider the following categories
of varieties.
\begin{enumerate}

\item $\mathcal{V}_k$: the category of smooth and 
projective varieties.

\item $\mathcal{V}'_k$: the category of projective varieties 
having at most global quotient singularities, i.e., 
those that can be written as a quotient of an object of $\mathcal{V}_k$ by a finite group.

\item $\mathcal{N}_k$: the category of projective normal varieties.

\item $\mathcal{P}_k$: the category of all projective varieties.

\end{enumerate}
We have the chain of inclusions
$$\mathcal{V}_k\subset \mathcal{V}_k'\subset \mathcal{N}_k\subset 
\mathcal{P}_k.$$

Given $N \in \mathcal{M}_k^{\eff}$ we define a functor $\omega_N : \mathcal{V}^{op}_k \to Vec_{\mathbb{Q}}$ by 
$$\omega_N(X)=Hom_{\mathcal{M}_k^{\eff}}(M(X),N), \; \text{for}\; X\in \mathcal{V}_k.$$ 
We thus have a functor 
$\omega:\mathcal{M}_k^{\eff} \to PSh(\mathcal{V}_k)$ given by 
$N \mapsto \omega_N$. 

\begin{theorem} 
\label{fully faithful}
The functor $\omega:\mathcal{M}^{\eff}_k \to PSh(\mathcal{V}_k)$ is fully faithful, i.e., for every $M, N \in \mathcal{M}^{\eff}_k$, the natural morphism 
\begin{equation}
\label{ff}
Hom_{\mathcal{M}_k^{\eff}}(M,N) \to Hom(\omega_M,\omega_N)
\end{equation} 
is bijective.
\end{theorem}

\begin{remark}
The statement of the theorem appears without proof in \cite[2.2]{ajs} and is also mentioned in \cite[p.~12]{t}. 
\end{remark}

\begin{lemma}
\label{lem:omega-faithful}
The functor $\omega$ is faithful.
\end{lemma}

\begin{proof}
To show that the map \eqref{ff} is injective, we may assume that 
$M=M(X)$ and $N=M(Y)$ for $X,\,Y \in \mathcal{V}_k$. 
In this case, 
\eqref{ff}
has a retraction given by
$\alpha \in  Hom (\omega_M, \omega_N) \mapsto \alpha(id_X)$. Hence 
it is injective.
\end{proof}

\begin{definition}
$\,$

\begin{enumerate}

\item 
The $pcdh$ topology on $\mathcal{P}_k$ is the Grothendieck topology  generated by the covering families of the form 
$(X'\xrightarrow{p_{X'}} X, Z \xrightarrow{p_{Z}} X)$ such that $p_{X'}$ is a proper morphism, $p_Z$ is a closed embedding and
$p_{X'}^{-1} (X - p_Z(Z)) \to X - p_Z(Z)$ is an isomorphism.
To avoid problems, we also add the empty family to the covers of the empty scheme.

\item
The $fh$ topology on $\mathcal{N}_k$ is the topology associated to the pretopology formed by the finite families $(f_i : Y_i \to X)_{i \in I}$ such that $\cup_i f_i : \coprod_{i \in I} Y_i \to X$ is finite and surjective.
\end{enumerate}
\end{definition}

\begin{lemma} 
\label{cdh}
Let $M \in \mathcal{M}_k^{\eff}$. The presheaf $\omega_M$ can be extended to a presheaf $\omega'_M$ on 
$\mathcal{V}'_k$ such that for $X = X'/G$ with $X' \in \mathcal{V}_k$ and $G$ a finite group, we have $\omega'_M (X)=\omega_M(X')^G$.
\end{lemma}

\begin{proof}
By \cite[Example 8.3.12]{ful}, we can define refined intersection class with rational coefficients which can be used to define a
category of effective Chow motives $\mathcal{M}'^{\eff}_k$. 
Moreover, the canonical functor 
$\phi : \mathcal{M}^{\eff}_k \to \mathcal{M}'^{\eff}_k$,
induced by the inclusion $\mathcal{V}_k \to \mathcal{V}'_k$,
is an equivalence of categories 
(cf.~\cite[Proposition 1.2]{ra}).
For $X\in \mathcal{V}'_k$, we set 
$$\omega'_M(X) = Hom_{\mathcal{M}'^{\eff}_k}(M(X),\phi(M)).$$
In this way we get a presheaf $\omega'_M$ on $\mathcal{V}'_k$
which extends the presheaf $\omega_M$. Moreover,
the identification $\omega'_M(X'/G)=\omega_M(X')^G$
is clear.
\end{proof}

\begin{lemma}
\label{lem:extension-pcdh}
Let $M\in \mathcal{M}^{\eff}_k$.
The presheaf $\omega_M$ can be uniquely extended to a 
$pcdh$-sheaf $\omega''_M$ on $\mathcal{P}_k$.
\end{lemma}

\begin{proof}
From \ref{subspace topology}(1) and the blow-up formula for Chow groups we deduce that 
$\omega_M$ is a $pcdh$-sheaf on $\mathcal{V}_k$. 
The result now follows from the first claim in
\ref{subspace topology}.
\end{proof}

\begin{lemma}
\label{extension}
Let $M\in \mathcal{M}_k^{\eff}$. We have $\omega''_M|_{\mathcal{V}'_k} \cong \omega'_M$. 
\end{lemma}

\begin{proof}
We will show that $\omega'_M$ extends uniquely to a 
$pcdh$-sheaf on $\mathcal{P}_k$. 
Since $\omega'_M|_{\mathcal{V}_k} \cong \omega_M$, 
\ref{lem:extension-pcdh}
shows that this extension is given
$\omega''_M$. In particular, we have
$\omega''_M|_{\mathcal{V}'_k} \cong \omega'_M$.

From the first statement in \ref{subspace topology}, it suffices to 
show that $\omega'_M$ is a $pcdh$-sheaf on 
$\mathcal{V}'_k$. 
To do so, we use \ref{subspace topology}(2).
Let $X\in \mathcal{V}_k$ and $G$ a finite group acting on $X$. 
Let $Z\subset X$ be a smooth closed subscheme globally invariant under 
$G$. Let $\tilde{X}$ be the blow-up of $X$ along $Z$ and
let $E$ be the exceptional divisor.
We need to show that 
$$\omega'_M(X/G)\simeq {\rm ker}\{\omega'_M(\tilde{X}/G)\oplus
\omega'_M(Z/G)\to \omega'_M(E/G)\}.$$
This is equivalent to 
$$\omega_M(X)^G\simeq {\rm ker}\{\omega_M(\tilde{X})^G\oplus
\omega_M(Z)^G\to \omega_M(E)^G\}.$$
This is true by the blow-up formula for Chow groups and 
the exactness of the functor $(-)^G$ on 
$\mathbb{Q}[G]$-modules.
\end{proof}

\begin{lemma}
\label{fh}
Let $M\in \mathcal{M}^{\eff}_k$. Then $\omega''_M|_{\mathcal{N}_k}$ is an $fh$-sheaf.
\end{lemma}

\begin{proof}
Let $X=Y/G$ with $Y \in \mathcal{N}_k$ and $G$ a finite group. 
We claim that $\omega_M''(Y)^G \cong \omega''_M(X)$. 
When $Y$ is smooth, this is true by \ref{cdh} and \ref{extension}.
In general, we will prove this by induction on the dimension of $Y$
and we will no longer assume that $Y$ is normal. (However, 
it is convenient to assume that $Y$ is reduced.) 
If $Y$ has dimension zero then $Y$ is smooth and the result is known.
Assume that $dim(Y)=d>0$. By $G$-equivariant resolution of singularities there is a blow-up square
$$\xymatrix{
E \ar[r]^g \ar[d]^h & Y' \ar[d]^f \\
Z \ar[r]^i  & Y}$$
such that $Y'$ is smooth, 
$Z\subset Y$ is a nowhere dense closed subscheme
which is invariant under the action of $G$ 
and such that 
$Y'-E\simeq Y-Z$. 
Taking quotients by $G$ gives the following blow-up square
$$\xymatrix{
E/G \ar[r] \ar[d] & Y'/G \ar[d] \\
Z/G \ar[r]  & Y/G.}$$
Using induction on dimension and the fact that 
$\omega''_M$ is a $pcdh$-sheaf, we are left to show that 
$\omega''_M(Y')^G\simeq \omega''_M(Y'/G)$. 
This follows from \ref{cdh} and
\ref{extension}.
\end{proof}

\begin{proof}[Proof of Theorem \ref{fully faithful}]
It remains to show that the functor is full. Let 
$M,\, N \in  \mathcal{M}_k^{\eff}$. 
Since every effective motive is a direct summand of the motive of 
a smooth and projective variety, 
we may assume that $M=M(X)$ and $N=M(Y)$ for $X, \, 
Y \in \mathcal{V}_k$.
Let $f : \omega_{M(X)} \to \omega_{M(Y)}$ be a morphism of presheaves. 
As in the proof of \ref{lem:omega-faithful},
there is an associated morphism of Chow motives 
$f_X(id_{M(X)}):M(X) \to M(Y)$. For 
$Z\in \mathcal{V}_k$ and 
$c\in \omega_{M(X)}(Z)$, 
we need to show that 
$f_X(id_{M(X)}) \circ c = f_Z(c) \in \omega_{M(Y)}(Z)$.

By \ref{extension} the morphism $f$ can be uniquely extended to a morphism $f'':\omega''_{M(X)} \to \omega''_{M(Y)}$ of $pcdh$-sheaves on $\mathcal{P}_k$. Moreover, by 
\ref{fh}, the restriction of $f''$ to 
$\mathcal{N}_k$ is a morphism of 
$fh$-sheaves.
By \cite[Prop 2.2.6]{a1} (see also \cite{Sing}), any $fh$-sheaf has canonical transfers 
and $f''|_{\mathcal{N}_k}$ commutes with them.  
Now $c \in  \omega_{M(X)}(Z) = CH^*(Z \times X)$ 
is the class of a finite correspondence
$\gamma\in Cor(Z,X)$ and 
$c=\omega'_{M(X)}(\gamma)(id_{M(X)})$.
(This follows from 
\cite[Corollary 19.2]{mvw} and the property that 
$C_*\mathbb{Q}_{tr}(X)$ is 
fibrant with respect to the projective motivic model structure;
this property holds because $X$ is proper, see 
\cite[Cor.~1.1.8]{a2}.)
Thus, we have:
$$f_Z(c)=f'_Z(\omega'_{M(X)}(\gamma)(id_{M(X)}))
=\omega'_{M(Y)}(\gamma)(f'_X(id_{M(X)}))=f_X(id_{M(X)}) \circ c.$$
This completes the proof.
\end{proof}

\section{}

\label{appendix-B}

Let $C'$ be a category and $\tau'$ a Grothendieck topology on $C'$. Given a functor $u : C \to C'$ there is an induced topology $\tau$ 
on $C$. (For the definition of the induced topology, we 
refer the reader to \cite[III 3.1]{sga4}.)

\begin{proposition} 
\label{induced topology}
Assume that $u:C \hookrightarrow C'$ is fully faithful and that  
every object of $C'$ can be covered, with respect to the topology $\tau'$, by objects in $u(C)$. 
Let $X\in C$ and $R\subset X$ be a sub-presheaf of $X$. Then, 
the following conditions are equivalent:
\begin{enumerate}
 
\item $R \subset X$ is a covering sieve for $\tau$.
   
\item 
There exists a family $(X_i \to X)_{i \in I}$ such that 
\begin{enumerate}

\item $R \supset Image(\coprod_i X_i \to X)$; 

\item $(u(X_i) \to u(X))_{i \in I}$  is a covering family for $\tau'$.   

\end{enumerate}
\end{enumerate}
Moreover $u_*:Shv(C')\to Shv(C)$ is an equivalence of categories.

\end{proposition}

\begin{proof}
The last assertion is just \cite[Th\'eor\`eme III.4.1]{sga4}.

$(1) \Rightarrow (2)$:
Suppose $R \subset X$ is a covering sieve for $\tau$. Then $u^*(R) \to u(X)$ is a bicovering morphism for $\tau'$, i.e., induces an isomorphism on the associated sheaves (see \cite[D\'efinition I.5.1, D\'efinition II.5.2 and Proposition III. 1.2]{sga4} where $u^*$ was denoted by $u_!$ which is not so standard nowadays). Since $C$ contains a generating set of objects for $\tau'$, there is a covering family of the form $(u(X_i) \to u(X))_{i \in I}$ for the topology 
$\tau'$ and a dotted arrow as below
$$\xymatrix{& u^*(R) \ar[d]\\
\coprod_i u(X_i) \ar[r]\ar@{.>}[ur] & u(X)}$$  
making the triangle commutative.

Now, recall that for $U'\in C'$, one has
$$u^*(R)(U')=\underset{(V,\,U'\to u(V))\, \in  \, U'\backslash C}{colim} \; R(V)$$
where $U'\backslash C$ is the comma category. 
Using the fact that $u:C\hookrightarrow C'$ is fully faithful, we see that for $U'=u(U)$ the category $u(U)\backslash C$ has an initial 
object given by $(U,\,id:u(U)=u(U))$. It follows that 
$R(U)\simeq u^*(R)(u(U))$ which can be also written as
$R\simeq u_*u^*(R)$.
In particular, the maps of presheaves $u(X_i)\to u^*(R)$ are uniquely induced 
by maps of presheaves $X_i \to R$. This shows that 
$R$ contains the image of the morphism of presheaves 
$\coprod_iX_i \to X$.

$(2) \Rightarrow (1)$:
Now suppose that condition (2) is satisfied. 
We must show that $u^*(R) \to u(X)$ is a bicovering morphism of 
presheaves for $\tau'$, i.e., that $a_{\tau'}(u^*(R)) \to a_{\tau'}(u(X))$ is an isomorphism where $a_{\tau'}$ is the ``associated $\tau'$-sheaf'' functor.

Since the surjective morphism of sheaves 
$a_{\tau'}(\coprod_i u(X_i)) \to a_{\tau'}(u(X))$ factors through 
$a_{\tau'}(u^*(R))$, the surjectivity of $a_{\tau'}(u^*(R)) \to a_{\tau'}(u(X))$ is clear. 
Since every object of $C'$ can be covered by objects in
$u(C)$, to prove injectivity it suffices to show that 
$u^*(R)(u(U))\to u(X)(u(U))$ is injective for all
$u\in C$.  
From the proof of the implication
$(1)\Rightarrow (2)$, we know that this map is nothing but the inclusion $R(U)\hookrightarrow X(U)$. This finishes the proof.
\end{proof}

The $pcdh$ topology on 
$\mathcal{P}_k$ induces topologies on 
$\mathcal{V}_k$ and $\mathcal{V}'_k$ which we also 
call $pcdh$. The next corollary gives a description 
of these topologies.

\begin{corollary} 
\label{subspace topology}
The categories of $pcdh$-sheaves on $\mathcal{V}_k$
and $\mathcal{V}_k'$ are equivalent to the category of 
$pcdh$-sheaves on $\mathcal{P}_k$. Moreover, 
$pcdh$-sheaves on $\mathcal{V}_k$ and $\mathcal{V}_k'$ 
can be characterized as follows.

\begin{enumerate}
\item A presheaf $F$ on $\mathcal{V}_k$ such that 
$F(\emptyset)=0$ is a $pcdh$-sheaf if 
and only if for every smooth and projective variety $X$, and 
every closed and smooth subscheme $Z\subset X$, one has
$$F(X)\simeq ker\{F(\tilde{X})\oplus F(Z) \to F(E)\}$$
where $\tilde{X}$ is the blow-up of $X$ in $Z$ and 
$E\subset \tilde{X}$ is the exceptional divisor. 

\item 
A presheaf $F$ on $\mathcal{V}'_k$ such that 
$F(\emptyset)=0$ is a $pcdh$-sheaf if 
and only if for every smooth and projective variety $X$ together 
with an action of a finite group $G$, 
and every closed and smooth subscheme $Z\subset X$ globally invariant under the action of $G$, one has
$$F(X/G)\simeq ker\{F(\tilde{X}/G)\oplus F(Z/G) \to F(E/G)\}$$
where $\tilde{X}$ is the blow-up of $X$ in $Z$ and
$E\subset \tilde{X}$ is the exceptional divisor. 

\end{enumerate}
\end{corollary}

\begin{proof}
By Hironaka's resolution of singularities,
every projective variety can be covered (with respect to the $pcdh$ topology) by smooth and projective varieties, i.e., by objects in 
the subcategory 
$\mathcal{V}_k$ (and hence $\mathcal{V}'_k$).
Thus, the first claim follows from
\cite[Th\'eor\`eme III.4.1]{sga4}.

Next, we only treat $(2)$ 
as the verification of $(1)$ is similar and in fact easier.
The condition in (2) is necessary for $F$ to be a 
$pcdh$-sheaf as 
$(\tilde{X}/G \to X/G, \, Z/G \to X/G)$ 
is a $pcdh$-cover.
Hence we only need to show that the condition is sufficient.

Let $X/G\in \mathcal{V}'_k$ where 
$X$ is a smooth and projective variety and
$G$ is a finite group acting on $X$.
It suffices to show that 
$F(X/G)\simeq Colim_{R\subset (X/G)}\;F(R)$
where $R\subset (X/G)$ varies among covering sieves for the 
$pcdh$-topology on $\mathcal{V}'_k$. 
We will prove a more precise statement namely: any
covering sieve $R\subset (X/G)$ can be refined into a covering sieve
$R'\subset (X/G)$ such that $F(X/G)\simeq F(R')$.

By \ref{induced topology},
there exists a $pcdh$-cover $(Y_i\to (X/G))_i$ 
with $Y_i\in \mathcal{V}'_k$ and such that 
$R\supset Image(\coprod_i Y_i \to (X/G))$.
Using equivariant resolution of singularities, we may find a 
sequence of equivariant blow-ups in smooth centers 
$Z_i\subset X_i$:
$$X_n\to \cdots \to X_1\to X_0=X$$
such that the covering family 
\begin{equation}
\label{eq:cover-fam-pcdh}
(X_n/G\to X/G,\,Z_{n-1}/G\to X/G,\, \cdots,\, Z_0/G\to X/G)
\end{equation}
is a refinement of the sieve $R$. 
Using induction and the property satisfied by $F$ from $(2)$, we see
that 
\begin{equation}
\label{eqn:induc-pcdh-sq}
F(X/G)\simeq ker\{F(X_n/G)\oplus F(Z_{n-1}/G)\oplus \cdots \oplus 
F(Z_0/G)\hspace{1cm}
\end{equation}
$$\hspace{6cm}\longrightarrow F(E_n/G)\oplus \cdots 
\oplus F(E_1/G)\}$$
where $E_i\subset X_i$ is the exceptional divisor of the blow-up
with center $Z_{i-1}$.
It is easy 
to deduce from \eqref{eqn:induc-pcdh-sq}
that $F(X/G)\simeq F(R')$
when $R'\subset (X/G)$ is the image of the
covering family
\eqref{eq:cover-fam-pcdh}.
\end{proof}

\end{document}